\documentclass[12pt,reqno]{amsart}

\usepackage{multirow}
\usepackage{hhline}

\usepackage{mathtools}
\usepackage{times}
\usepackage[T1]{fontenc}
\usepackage{mathrsfs}
\usepackage{latexsym}
\usepackage[dvips]{graphics}
\usepackage[titletoc, title]{appendix}
\usepackage{amsmath,amsfonts,amsthm,amssymb,amscd}
\usepackage[dvipsnames]{xcolor}
\usepackage{hyperref}
\usepackage{amsmath}

\usepackage{color}
\usepackage{breakurl}

\usepackage{comment}
\newcommand{\bburl}[1]{\textcolor{blue}{\url{#1}}}
\newcommand{\seqnum}[1]{\href{https://oeis.org/#1}{\rm \underline{#1}}}

\newtheorem{thm}{Theorem}[section]

\newtheorem{cor}[thm]{Corollary}

\newtheorem{prop}[thm]{Proposition}
\newtheorem{exa}[thm]{Example}

\newtheorem{defi}[thm]{Definition}
\newtheorem{rek}[thm]{Remark}

\numberwithin{equation}{section}

\DeclareFontFamily{U}{mathx}{}
\DeclareFontShape{U}{mathx}{m}{n}{<-> mathx10}{}
\DeclareSymbolFont{mathx}{U}{mathx}{m}{n}
\DeclareMathAccent{\widehat}{0}{mathx}{"70}
\DeclareMathAccent{\widecheck}{0}{mathx}{"71}

\begin{document}

\title{Approximation by Egyptian fractions and the weak greedy algorithm}

\author{H\`ung Vi\d{\^e}t Chu}

\email{\textcolor{blue}{\href{mailto:hungchu2@illinois.edu}{hungchu2@illinois.edu}}}
\address{Department of Mathematics, University of Illinois at Urbana-Champaign, Urbana, IL 61820, USA}

\begin{abstract} 
Let $0 < \theta \leqslant 1$. A sequence of positive integers $(b_n)_{n=1}^\infty$ is called a weak greedy approximation of $\theta$ if $\sum_{n=1}^{\infty}1/b_n = \theta$. We introduce the weak greedy approximation algorithm (WGAA), which, for each $\theta$, produces two sequences of positive integers $(a_n)$ and $(b_n)$ such that 
\begin{enumerate}
\item[a)] $\sum_{n=1}^\infty 1/b_n = \theta$;
\item[b)] $1/a_{n+1} < \theta - \sum_{i=1}^{n}1/b_i < 1/(a_{n+1}-1)$ for all $n\geqslant 1$;
\item[c)] there exists $t\geqslant 1$ such that $b_n/a_n \leqslant t$ infinitely often.
\end{enumerate}
We then investigate when a given weak greedy approximation $(b_n)$ can be produced by the WGAA. Furthermore, we show that for any non-decreasing $(a_n)$ with $a_1\geqslant 2$ and $a_n\rightarrow\infty$, there exist $\theta$ and $(b_n)$ such that a) and b) are satisfied; whether c) is also satisfied depends on the sequence $(a_n)$. Finally, we address the uniqueness of $\theta$ and $(b_n)$ and apply our framework to specific sequences. 
\end{abstract}



\thanks{The author would like to thank the anonymous referees for a careful reading of the paper and helpful comments.}

\maketitle

\section{Introduction}

Throughout this paper, let $\theta$ denote a number in $(0,1]$ and let $G: (0,1]\rightarrow \mathbb{N}_{\geqslant 2}$ be the function 
$$G(\theta)\ =\ \left\lfloor \frac{1}{\theta}\right\rfloor +1;$$
that is, $G(\theta)$ gives the unique positive integer $a\geqslant 2$ such that 
$$\frac{1}{a}\ <\ \theta \ \leqslant\ \frac{1}{a-1}.$$

An \textit{Egyptian} fraction is a fraction of the form $1/n$ for some positive integer $n$.
We consider the problem of representing $\theta$ as an infinite sum of Egyptian fractions. One natural method is the \textit{greedy underapproximation algorithm} (GUA), which constructs a sequence of positive integers $(a_n)_{n=1}^\infty$ recursively as follows: $a_1 = G(\theta)\geqslant 2$; supposing that $a_1, \ldots, a_n$ have been constructed, let 
$$a_{n+1}\ =\ G\left(\theta-\sum_{i=1}^n \frac{1}{a_i}\right).$$
By \cite[(3)]{Na23}, the sequence $(a_n)$ is strictly increasing and particularly, satisfies 
\begin{equation}\label{e1}a_1\geqslant 2\mbox{ and } a_{n+1}\ \geqslant\ a_n^2-a_n+1.\end{equation}
Since by construction,
$$\theta-\sum_{i=1}^n\frac{1}{a_n} \ \leqslant\ \frac{1}{a_{n+1}-1}\ \rightarrow\ 0,$$
we have
$$\sum_{n=1}^\infty \frac{1}{a_n} = \theta.$$
According to \cite[Theorem 5]{Na23}, if $\theta = p/q$, where $p, q$ are positive integers such that $p$ divides $q+1$, then the GUA produces the best approximations; i.e., the $n$-term approximation $\sum_{i=1}^n1/a_i$ outperforms any other $n$-term underapproximations using Egyptian fractions. This generalizes a result in \cite{Cu22, So05, Ta21}. The proof involves an useful inequality established in \cite{AB15} (see also \cite{Na22}.) However, such optimality does not hold for general $\theta$ (see \cite[Section 5]{Na23}.)

The goal of this paper is to investigate a weak version of the GUA, which is inspired by the so-called \textit{(weak) thresholding greedy algorithm} (TGA) in the area of functional analysis. We describe the (weak) TGA briefly. Let $X$ be an infinite-dimensional, complete, normed vector space. Assume further that $X$ has a basis $\mathcal{B} = (e_n)_{n=1}^\infty$ so that every vector $x\in X$ can be represented by a formal series $\sum_{n=1}^\infty a_n e_n$, where $a_n$ are scalars. (The series converges to a vector when our basis is Schauder; however, for general Markushevich bases, the series may only be formal.)  In order to form an $m$-term approximation of $x$, the TGA chooses $m$ largest coefficients $a_n$ in modulus. Formally, let $A\subset\mathbb{N}$ verify $|A| = m$ and 
\begin{equation}\label{e2}\min_{n\in A}|a_n| \ \geqslant\ \max_{n\notin A}|a_n|.\end{equation} Then the TGA produces the $m$-term approximation $\sum_{n\in A}a_n e_n$. It is not always true that approximations produced by this method converge to the original vector $x$ as $m$ grows. In fact, Konyagin and Temlyakov \cite{KT99} called a basis \textit{quasi-greedy} if these approximations converge to the desired $x$. Meanwhile, Temlyakov \cite{Te98} introduced a weaker version of the TGA, called the weak TGA (WTGA), which is more flexible in forming approximating sums. In particular, fixing a number $t\in (0,1]$, the WTGA considers sets $A$ satisfying $|A| = m$ and 
\begin{equation}\label{e3}\min_{n\in A}|a_n|\ \geqslant\ t\max_{n\notin A}|a_n|.\end{equation}
Clearly, \eqref{e3} is weaker than \eqref{e2}. In other words, the WTGA chooses the ``largest coefficients up to a constant." Surprisingly, the flexibility of the WTGA does not affect convergence: a basis is quasi-greedy under the TGA if and only if it is quasi-greedy under the WTGA (see \cite[Section 1.5]{Te08}.)

Inspired by the aforementioned interactions between the TGA and the WTGA, we introduce the \textit{weak greedy approximation algorithm} (WGAA) as a companion of the GUA. The idea is that at the $n$th step of our weak algorithm, we pick $a_n$ based on the ``greedy choice up to a constant". 
Specifically, fix $t\in \mathbb{R}_{\geqslant 1}$ and an infinite set $\Lambda\subset \mathbb{N}$.
For each $\theta\in (0,1]$, we define the $(t, \Lambda)$-WGAA as follows: let $a_1=G(\theta)$. Choose $b_1\geqslant a_1$. Additionally, we require $b_1\leqslant ta_1$ if $1\in \Lambda$. Assuming that $a_1, b_1, \ldots, a_n, b_n$ have been defined, we let
\begin{equation}\label{e4}a_{n+1} \ =\ G\left(\theta - \sum_{i=1}^{n}\frac{1}{b_i}\right).\end{equation}
Choose $b_{n+1}\geqslant a_{n+1}$. Additionally, we require $b_{n+1}\leqslant ta_{n+1}$ if $n+1\in \Lambda$. We see that the $(t, \Lambda)$-WGAA generalizes the GUA by simply setting $t = 1$ and $\Lambda = \mathbb{N}$. 

\begin{defi}\normalfont
An infinite sequence of positive integers $(b_n)_{n=1}^\infty$ is called a \textit{weak greedy approximation} of $\theta$ if 
$\sum_{n=1}^\infty 1/b_n = \theta$ and for all $n\geqslant 1$, 
\begin{equation}\label{e6}G\left(\theta-\sum_{i=1}^{n-1} \frac{1}{b_i}\right) \ \leqslant\ b_n.\end{equation}
\end{defi}
Inequality \eqref{e6} indicates that a term $b_n$ is not necessarily picked by the greedy algorithm. Attentive readers may notice that  \eqref{e6} is superfluous. Indeed, suppose that for some $N$, 
$$b_N \ <\ G\left(\theta-\sum_{i=1}^{N-1} \frac{1}{b_i}\right)\ =:\ a_N.$$
Then 
$$\sum_{i=1}^N \frac{1}{b_i}\ =\ \sum_{i=1}^{N-1}\frac{1}{b_i} + \frac{1}{b_N}\ \geqslant\ \sum_{i=1}^{N-1}\frac{1}{b_i} + \frac{1}{a_N-1}\ \geqslant \ \theta,$$
which contradicts $\sum_{n=1}^\infty 1/b_n = \theta$.

We describe the paper's structure. In Section \ref{conv}, we show that the WGAA satisfies the minimal requirement for an algorithm to be sensible; that is, for every $\theta$, the sequence $(b_n)$ produced by the WGAA satisfies 
\begin{equation}\label{e5}\sum_{n=1}^\infty \frac{1}{b_n}  \ =\ \theta.\end{equation}
This is the analog of the relation between the TGA and the WTGA. 
Moreover, we compute the growth rate of the sequence $(b_n)$ produced by the $(t, \Lambda)$-WGAA when $\Lambda = \mathbb{N}$ and $b_n = \lceil ta_n\rceil$ (Proposition \ref{p1}.)

In Section \ref{gen}, we carry out a deeper study of the two sequences $(a_n)$ and $(b_n)$ produced by the WGAA. According to Section \ref{conv}, if $(b_n)$ is produced by the WGAA applied to $\theta$, then $(b_n)$ is a weak greedy approximation of $\theta$. We shall show that the converse is not true: there exist $\theta$ and $(b_n)$ such that $\sum_{n=1}^\infty 1/b_n = \theta$, but $(b_n)$ cannot be produced by the WGAA.
To do so, we observe that $\theta, (a_n), (b_n)$ produced by the WGAA have three properties
\begin{enumerate}
\item[a)] $\sum_{n=1}^\infty 1/b_n = \theta$ (see Section \ref{conv});
\item[b)] \eqref{e4} holds for all $n$; 
\item[c)] there exists $t\geqslant 1$ such that $b_n/a_n \leqslant t$ infinitely often.
\end{enumerate}
As we shall see, condition c) guarantees the convergence \eqref{e5}. However, even when $\theta$, $(a_n)$, and $(b_n)$ verify a) and b), they do not necessarily satisfy c). As a result, in such cases, $(b_n)$ cannot be produced by the WGAA. We then go further to characterize the situation when c) does not hold (see Proposition \ref{p2}.) 

Next, we consider the following question: given a non-decreasing sequence $(a_n)$ with $a_1\geqslant 2$ and $a_n\rightarrow\infty$, are there $\theta\in (0,1]$ and $(b_n)$ such that a) and b) hold? According to \cite[Corollary 3]{Na23}, the answer is positive if $a_{n+1}\geqslant a_n^2-a_n+1$, in which case, $\theta = \sum_{n=1}^\infty 1/a_n$ and $b_n = a_n$ for all $n\geqslant 1$. By explicit construction, we answer the aforementioned question in the affirmative for any non-decreasing sequence $(a_n)$ with $a_1\geqslant 2$ and $a_n\rightarrow\infty$ (see Theorem \ref{m1} and its Corollary \ref{c2}.)

Section \ref{unique} gives necessary and sufficient conditions for when a sequence $(a_n)$ gives unique $\theta$ and $(b_n)$ (Corollary \ref{c3} and Proposition \ref{p4}.) Finally, Section \ref{seq} applies the framework from previous sections to particular sequences including geometric progressions, arithmetic progressions, and the Fibonacci sequence.

\section{Convergence of the WGAA}\label{conv}
The minimal requirement we want the WGAA to satisfy is convergence, which is confirmed by the following proposition.

\begin{prop}\label{p5}
If $(b_n)_{n=1}^\infty$ is obtained from the $(t, \Lambda)$-WGAA applied to $\theta$, then 
$$\sum_{n=1}^\infty \frac{1}{b_n} \ =\ \theta.$$
\end{prop}

\begin{proof}
Let $(a_n), (b_n)$ be the two sequences produced by the $(t, \Lambda)$-WGAA applied to $\theta$: 
for each $n\geqslant 1$, 
\begin{equation}\label{e7}a_n \ =\ G\left(\theta - \sum_{i=1}^{n-1}\frac{1}{b_i}\right); \mbox{ equivalently, }0 \ <\ \frac{1}{a_n}\ <\ \theta - \sum_{i=1}^{n-1}\frac{1}{b_i}\ \leqslant\ \frac{1}{a_n-1}.\end{equation}
Hence, $(a_n)$ is non-decreasing. It suffices to prove that $(a_n)$ is unbounded. Suppose otherwise that there is some $M$ such that $a_n\leqslant M$ for all $n$. Then $b_n\leqslant Mt$ infinitely often, which implies that $\sum_{n=1}^\infty 1/b_n = \infty$, contradicting \eqref{e7}.
\end{proof}

Next, we consider a special case of the general $(t, \Lambda)$-WGAA by requiring that $\Lambda = \mathbb{N}$ and for all $n$, $b_n = \lceil ta_n\rceil$. Let us denote this algorithm by $\mathcal{G}(t)$. 
Suppose that we use $\mathcal{G}(t)$ to obtain an $n$-term approximation $\sum_{i=1}^n 1/c_i$ of $\theta$. Then a logical choice is to have $c_i = b_i = \lceil ta_i\rceil$ for all $1\leqslant i\leqslant n-1$, while $c_n = a_n$. (It makes no sense if we do not choose the last term $c_n$ greedily.) An approximation by $\mathcal{G}(4/3)$ may outperform the GUA. We borrow an example from \cite{Na23}. The GUA gives $1/3 + 1/17$ as a $2$-term underapproximation of $19/48$, while $\mathcal{G}(4/3)$ gives $1/4+1/7$. We have
$$\frac{1}{3} + \frac{1}{17}\ < \ \frac{1}{4} + \frac{1}{7} \ < \ \frac{19}{48}.$$

By definition, $\mathcal{G}(1)$ is the greedy underapproximation algorithm. There is an interesting difference between $t = 1$ and $t > 1$. If $(b_n)$ is obtained by $\mathcal{G}(1)$ applied to $\theta$, then \cite[(3)]{Na23} gives $$\frac{b_{n+1}}{b_n}\ \geqslant\  b_n-1+\frac{1}{b_n}.$$
Since $\lim_{n\rightarrow\infty}b_n = \infty$, we get $\lim_{n\rightarrow\infty} b_{n+1}/b_n = \infty$. However, the limit is finite when $t > 1$ as the following proposition shows. 

\begin{prop}\label{p1}
If $(b_n)_{n=1}^\infty$ is the sequence from $\mathcal{G}(t)$ applied to $\theta$, then 
$$\lim_{n\rightarrow\infty} \frac{b_{n+1}}{b_n} \ =\ \begin{cases}t/(t-1)&\mbox{ if }t > 1,\\ \infty &\mbox{ if 
 }t = 1.\end{cases}$$
\end{prop}

Before proving Proposition \ref{p1}, we record an important inequality addressing the relation between $(a_n)$ and $(b_n)$ produced by the WGAA. For each $n\geqslant 1$, we have
\begin{align*}
\frac{1}{a_{n+1}}\ <\ \theta - \sum_{i=1}^n\frac{1}{b_i}\ =\ \left(\theta - \sum_{i=1}^{n-1}\frac{1}{b_i}\right) - \frac{1}{b_n}\ \leqslant\ \frac{1}{a_n-1}-\frac{1}{b_n},
\end{align*}
and
\begin{align*}
\frac{1}{a_{n+1}-1}\ \geqslant\ \theta - \sum_{i=1}^n \frac{1}{b_i}\ =\ \left(\theta - \sum_{i=1}^{n-1}\frac{1}{b_i}\right) - \frac{1}{b_n}\ >\ \frac{1}{a_n}- \frac{1}{b_n}.
\end{align*}
Hence,
\begin{equation}\label{e8}\frac{1}{a_n} - \frac{1}{a_{n+1}-1}\ <\ \frac{1}{b_n}\ <\ \frac{1}{a_n-1}-\frac{1}{a_{n+1}},\forall n\in\mathbb{N}.\end{equation}

\begin{proof}[Proof of Proposition \ref{p1}]
The case $t=1$ is explained right before Proposition \ref{p1}. Let $t > 1$. The right side of \eqref{e8} yields
$$
\frac{1}{a_{n+1}}\ <\ \frac{1}{a_n-1}-\frac{1}{b_n}\ =\ \frac{1}{a_n-1}-\frac{1}{\lceil ta_n\rceil}\ <\ \frac{1}{a_n-1}-\frac{1}{ta_n+1},\forall n\in \mathbb{N}.
$$
Therefore,
\begin{equation}\label{e9}\frac{1}{a_{n+1}}\ <\ \frac{(t-1)a_n+2}{(ta_n+1)(a_n-1)}\ \Longrightarrow\ \frac{a_{n+1}}{a_n}\ >\ \frac{\left(t+\frac{1}{a_n}\right)\left(1-\frac{1}{a_n}\right)}{t-1+\frac{2}{a_n}}.\end{equation}
The left side of \eqref{e8} yields
$$
\frac{1}{a_{n+1}-1} \ >\ \frac{1}{a_n}-\frac{1}{b_n}\ =\ \frac{1}{a_n} - \frac{1}{\lceil ta_n\rceil}\ \geqslant\ \frac{1}{a_n}-\frac{1}{ta_n}.
$$
Hence,
\begin{equation}\label{e10}
\frac{a_{n+1}}{a_n}\ <\ \frac{t}{t-1}+\frac{1}{a_n}.
\end{equation}
From \eqref{e9} and \eqref{e10}, we obtain that $\lim_{n\rightarrow\infty}a_{n+1}/a_n = t/(t-1)$. Since $b_n = \lceil ta_n\rceil$, we have the desired conclusion.
\end{proof}

\section{The range of the WGAA}\label{gen}
In this section, we address the question of whether every weak greedy approximation can be obtained from the WGAA.
The boundedness condition on the WGAA requires that for some $t\geqslant 1$, $b_n/a_n\leqslant t$ infinitely often, which  guarantees the convergence of $\sum_{n=1}^\infty 1/b_n$ to the desired $\theta$ (see the proof of Proposition \ref{p5}.) However, there exist $\theta$ and $(b_n)$ such that if $(a_n)$ satisfies \eqref{e4}, then 
$\lim_{n\rightarrow\infty}b_n/a_n = \infty$. By studying such a situation, we know more about the sequence $(a_n)$ (see Corollary \ref{c1}.) First, consider the following example. 

\begin{exa}\normalfont\label{ex1}
For $n\in \mathbb{N}$, let $b_n = n(n+2)$ and $\theta = 3/4$. It is easy to check that $\sum_{n=1}^\infty 1/b_n = \theta$. We claim that if $(a_n)$ satisfies \eqref{e4}, then $a_n = n+1$.  Indeed, it suffices to show that
$$\left\lfloor \left(\frac{3}{4}-\sum_{i=1}^{n-1} \frac{1}{i(i+2)}\right)^{-1}\right\rfloor \ =\ n, \forall n\in\mathbb{N}.$$
We have
\begin{align*}\left\lfloor \left(\frac{3}{4}-\sum_{i=1}^{n-1} \frac{1}{i(i+2)}\right)^{-1}\right\rfloor&\ =\ \left\lfloor \left(\sum_{i=1}^{\infty}\frac{1}{i(i+2)}-\sum_{i=1}^{n-1} \frac{1}{i(i+2)}\right)^{-1}\right\rfloor\\
 &\ = \ \left\lfloor \left(\sum_{i = n}^{\infty}\frac{1}{i(i+2)}\right)^{-1}\right\rfloor\\
 &\ =\ \left\lfloor \left(\frac{1}{2}\left(\frac{1}{n} + \frac{1}{n+1}\right)\right)^{-1}\right\rfloor\mbox{ by telescoping}\\
 &\ =\ \left\lfloor n + \frac{n}{2n+1}\right\rfloor\ =\ n.
\end{align*}
Hence, $a_n = n+1$ and $b_n/a_n\rightarrow\infty$. 
\end{exa}

The sequences $(a_n)$ and $(b_n)$ in Example \ref{ex1} do not have $b_n/a_n$ infinitely often bounded. In other words, a weak greedy approximation does not necessarily come from the WGAA. The next proposition provides a characterization of this situation.

\begin{prop}\label{p2}
Let $(b_n)_{n=1}^\infty$ be a weak greedy approximation of $\theta$ and $(a_n)_{n=1}^\infty$ satisfy \eqref{e4}. The following are equivalent
\begin{enumerate}
\item[i)] for all $t\geqslant 1$, $\{n: b_n/a_n\leqslant t\}$ is finite.
\item[ii)] $\lim_{n\rightarrow\infty}a_{n+1}/a_n = 1$. 
\end{enumerate}
\end{prop}

\begin{cor}\label{c1}
Let $(b_n)_{n=1}^\infty$ be a weak greedy approximation of $\theta$ and $(a_n)_{n=1}^\infty$ satisfy \eqref{e4}.
Then $(a_n)_{n=1}^\infty$ and $(b_n)_{n=1}^\infty$ are obtained from the WGAA if and only if for some $\varepsilon > 0$, $a_{n+1} > (1+\varepsilon) a_n$ infinitely often. 
\end{cor}

\begin{proof}[Proof of Proposition \ref{p2}]
i) $\Longrightarrow$ ii): Since $(a_n)$ is non-decreasing, it suffices to show that for all $\varepsilon > 0$, there exists $N$ such that $a_{n+1}/a_n < 1+\varepsilon$ for all $n > N$. Choose $M$ sufficiently large such that $M/(M-1) < 1+\varepsilon/2$. By i), there exists $N$ such that for all $n>N$, $b_n > Ma_n$ and $1/a_n < \varepsilon/2$.
By \eqref{e8}, 
$$
\frac{1}{a_{n+1}-1}\ >\ \frac{1}{a_n} - \frac{1}{b_n} \ >\ \frac{1}{a_n} - \frac{1}{Ma_n}\ =\ \frac{M-1}{M}\frac{1}{a_n}, \forall n > N,
$$
which gives
$$\frac{a_{n+1}}{a_n} \ <\ \frac{M}{M-1} + \frac{1}{a_n}\ <\ 1+\varepsilon, \forall n > N.$$

ii) $\Longrightarrow$ i): We prove by contrapositive. Choose $t\geqslant 1$ and suppose that $b_n/a_n\leqslant t$ infinitely often. Let $A$ be the infinite set $\left\{n: b_n/a_n \leqslant t\right\}$. By \eqref{e8}, we have
$$\frac{1}{a_{n+1}} \ <\ \frac{1}{a_n-1}-\frac{1}{b_n}\ \leqslant\ \frac{1}{a_n-1} - \frac{1}{ta_n}, \forall n\in A.$$
Trivial calculations give
$$\frac{a_{n+1}}{a_n}\ >\ \frac{t(a_n-1)}{ta_n - (a_n-1)}\ =\ \frac{a_n-1}{(a_n-1) - \frac{a_n-1}{t}+1}\ =\ \frac{1}{1-\frac{1}{t} + \frac{1}{a_n-1}}, \forall n\in A.$$
If $t = 1$, then $a_{n+1}/a_n > a_n-1$ for all $n\in A$. That $a_n \rightarrow\infty$ implies that $a_{n+1}/a_n \geqslant 2$ infinitely often, making ii) fail. If $t > 1$, choose $N$ sufficiently large such that for $n>N$, $a_n > 2t+1$. Then for all $n\in A$ and $n > N$,
$$\frac{a_{n+1}}{a_n}\ >\ \frac{1}{1-\frac{1}{t}+\frac{1}{2t}}\ =\ \frac{1}{1-\frac{1}{2t}},$$
which contradicts ii).
\end{proof}

\begin{rek}\normalfont
If we replace the hypothesis ``$\sum_{n=1}^\infty 1/b_n = \theta$" in Proposition \ref{p2} by ``$\sum_{n=1}^\infty 1/b_n < \theta$", both i) and ii) in Proposition \ref{p2} hold. Indeed, if $\theta-\sum_{n=1}^\infty 1/b_n =: c > 0$, then
$$a_n\ :=\ G\left(\theta-\sum_{i=1}^{n-1} \frac{1}{b_i}\right)\ \leqslant\ G(c),$$
so $(a_n)$ is bounded. 
\end{rek}

We state and prove the last result in this section.

\begin{thm}\label{m1}
Let $(a_n)_{n=1}^\infty\subset\mathbb{N}$ be non-decreasing such that $a_1\geqslant 2$ and $a_n\rightarrow\infty$. There exist $\theta\in (0,1)$ and $(b_n)_{n=1}^\infty$ such that $$\sum_{n=1}^\infty \frac{1}{b_n} \ =\ \theta,$$ 
and for every $n\geqslant 1$, $$a_n \ =\ G\left(\theta-\sum_{i=1}^{n-1}\frac{1}{b_i}\right).$$
\end{thm}

\begin{proof}
Since $a_n\rightarrow \infty$, we can form the infinite set $A\subset\mathbb{N}$ such that $n\in A$ if and only if $a_{n+1}-a_n\geqslant 1$. In other words, $A$ contains all the indices immediately before a jump in $(a_n)$. Write $A = \{n_1, n_2, n_3, \ldots\}$, where $n_1 < n_2 < n_3 < \cdots$. Note that $a_{n_j} < a_{n_{j+1}}$ for all $j$. We obtain the sequence $(b_n)$ by first constructing all the $b_{n}$ for $n\in A$ then constructing the rest.  

Step 1: for each $j\geqslant 1$, choose $b_{n_j}$ such that
\begin{equation}\label{e11}\frac{a_{n_j}a_{n_{j+1}}}{a_{n_{j+1}}-a_{n_j}}-1-\frac{2a_{n_j}-1}{a_{n_{j+1}}-a_{n_j}}\ <\ b_{n_j}\ <\ \frac{a_{n_j}a_{n_{j+1}}}{a_{n_{j+1}}-a_{n_j}},\end{equation}
which can be done since the distance between the two ends are greater than $1$. Note that \eqref{e11} is equivalent to
\begin{equation}\label{e17}\frac{1}{a_{n_j}} - \frac{1}{a_{n_{j+1}}}\ <\ \frac{1}{b_{n_j}}\ <\ \frac{1}{a_{n_j}-1}-\frac{1}{a_{n_{j+1}}-1}.\end{equation}
It follows that for each $j\geqslant 1$, 
\begin{equation}\label{e12}\frac{1}{a_{n_j}}\ <\ \sum_{i=j}^\infty\frac{1}{b_{n_i}}\ <\ \frac{1}{a_{n_j}-1}.\end{equation}

Step 2: Due to \eqref{e12}, we can choose a sequence of positive  numbers $(\theta_j)_{j=1}^\infty$ satisfying
$$\frac{1}{a_{n_j}}\ <\ \sum_{i=j}^\infty\frac{1}{b_{n_i}} + \theta_j\ <\ \frac{1}{a_{n_j}-1}.$$
Let $n_0 = 0$. For each $j\geqslant 1$, set $b_{n_{j-1}+1}= b_{n_{j-1}+2}= \cdots = b_{n_j-1} = N_j$, where $N_j$ is sufficiently large such that
$$\frac{n_j-n_{j-1}-1}{N_j}\ <\ \min\left\{\frac{\theta_1}{2^j}, \frac{\theta_2}{2^{j-1}}, \ldots, \frac{\theta_j}{2}\right\}.$$

Step 3: Set $\theta := \sum_{n=1}^\infty \frac{1}{b_n}$. We claim that $\theta\in (0,1)$. We have
\begin{align*}\sum_{n=1}^\infty\frac{1}{b_n}&\ =\ \sum_{j=1}^\infty\frac{1}{b_{n_j}} + \sum_{j=1}^\infty\sum_{i=n_{j-1}+1}^{n_j-1}\frac{1}{b_i}\\
&\ =\ \sum_{j=1}^\infty\frac{1}{b_{n_j}} + \sum_{j=1}^\infty \frac{n_j-n_{j-1}-1}{N_j}\\
&\ <\ \sum_{j=1}^\infty\frac{1}{b_{n_j}} + \sum_{j=1}^\infty \frac{\theta_1}{2^j} \\
& \ =\ \sum_{j=1}^\infty\frac{1}{b_{n_j}} + \theta_1\ <\ \frac{1}{a_{n_1}-1}\ \leqslant\ 1.
\end{align*}

Step 4: Finally, we need to verify that 
$$\frac{1}{a_n} \ <\ \sum_{i=n}^\infty \frac{1}{b_i}\ \leqslant\ \frac{1}{a_n-1},\forall n\geqslant 1.$$
Fix $n\geqslant 1$ and choose $j$ such that $n_{j-1} < n\leqslant n_j$. By \eqref{e12}, we have
$$\sum_{i=n}^\infty \frac{1}{b_i} \ \geqslant\ \sum_{i=n_j}^\infty \frac{1}{b_i}\ \geqslant\ \sum_{i=j}^\infty \frac{1}{b_{n_i}}\ >\ \frac{1}{a_{n_j}}\ =\ \frac{1}{a_n}.$$
On the other hand,
\begin{align*}
\sum_{i=n}^\infty \frac{1}{b_i}&\ \leqslant\ \sum_{i=j}^\infty \frac{1}{b_{n_i}} + \sum_{i=j}^\infty\sum_{n_{i-1}+1}^{n_i-1}\frac{1}{b_n}\\
&\ =\ \sum_{i=j}^\infty \frac{1}{b_{n_i}} + \sum_{i=j}^\infty\frac{n_i-n_{i-1}-1}{N_i}\\
&\ <\ \sum_{i=j}^\infty \frac{1}{b_{n_i}} + \sum_{i=j}^\infty \frac{\theta_j}{2^{i+1-j}}\\
&\ =\ \sum_{i=j}^\infty \frac{1}{b_{n_i}} + \theta_j\ <\ \frac{1}{a_{n_j}-1}\ =\ \frac{1}{a_n-1}.
\end{align*}
This completes our proof.
\end{proof}

\begin{cor}\label{c2}
Let $(a_n)_{n=1}^\infty\subset\mathbb{N}$ be non-decreasing with $a_1\geqslant 2$ and $a_n\rightarrow\infty$. Then $\lim_{n\rightarrow\infty} a_{n+1}/a_n \neq 1$ is equivalent to the existence of $\theta\in (0,1)$ and $(b_n)_{n=1}^\infty$ such that $(a_n)_{n=1}^\infty$ and $(b_n)_{n=1}^\infty$ are the sequences obtained from the WGAA applied to $\theta$. 
\end{cor}

\begin{proof}
Use Proposition \ref{p2} and Theorem \ref{m1}.
\end{proof}

\begin{rek}\normalfont
Observe that \eqref{e17} is stronger than \eqref{e8}. This observation is important in studying the uniqueness of $\theta$ and $(b_n)$ in the next section. 
\end{rek}

\section{Uniqueness of $\theta$ and $(b_n)$}\label{unique}
Thanks to Theorem \ref{m1}, we know the existence of $\theta$ and $(b_n)$ given any non-decreasing sequence $(a_n)$ with $a_1\geqslant 2$ and $a_n\rightarrow\infty$. We now give sufficient and necessary conditions for when $(a_n)$ determines $\theta$ and $(b_n)$ uniquely. By Step 2 in the proof of Theorem \ref{m1}, a necessary condition is that $(a_n)$ must be strictly increasing. We can then eliminate Step 2 in constructing the sequence $(b_n)$ because $A =\mathbb{N}$. We claim further that $a_{n+1}-a_n \geqslant 2$ for all $n\in \mathbb{N}$. Indeed, suppose $a_{N+1}-a_N = 1$ for some $N$. We rewrite \eqref{e11} as
\begin{equation}\label{e16}a_{N+1}a_N-2a_N\ \leqslant\ b_N\ \leqslant\ a_Na_{N+1}.\end{equation}
There are at least $2a_N+1$ choices of $b_N$, so $\theta$ and $(b_n)$ are not unique. (Note that we allow equalities in \eqref{e16} because the construction in the proof of Theorem \ref{m1} still works if we allow equalities in finitely many \eqref{e11}.) 

Moreover, $(b_n)$ must satisfy \eqref{e8}. The following proposition tells us precisely when \eqref{e8} determines $(b_n)$ unequivocally. 

\begin{prop}\label{p3}
Let $(a_n)_{n=1}^\infty$ be non-decreasing such that $a_1\geqslant 2$ and $a_n\rightarrow\infty$. Then $(b_n)_{n=1}^\infty$ is uniquely determined by \eqref{e8} if and only if
$$a_{n+1}-2\ \geqslant\ a_n\ \geqslant\ 2, \forall n\geqslant 1,$$ and for each $n$, 
one of the following holds
\begin{enumerate}
\item[i)] $a_{n+1}-a_n-1$ divides $a_n^2$, and 
$$a_{n+1} \ \geqslant\ \frac{\sqrt{3}}{2}\sqrt{4a_n^2-4a_n+3}+2a_n-\frac{1}{2};$$ 
\item[ii)] $a_{n+1}-a_n-1$ does not divide $a_n^2$, and 
$$
\left\lfloor\frac{a_n^2}{a_{n+1}-a_n-1}\right\rfloor\ \leqslant\ \frac{(a_n-1)^2}{a_{n+1}-a_n+1}.
$$
\end{enumerate}
\end{prop}

\begin{proof}[Proof of Theorem \ref{m1}]
By \eqref{e8}, $(b_n)$ is uniquely determined if and only if each of the intervals 
$$I_n\ :=\ \left(\frac{(a_n-1)a_{n+1}}{a_{n+1}-a_n+1}, \frac{a_n(a_{n+1}-1)}{a_{n+1}-a_n-1}\right)$$
contains exactly one positive integer. It is easy to verify that there always exists one largest integer in $I_n$, called $k_n$. In order that $I_n$ contains no other integers, we need
\begin{equation}\label{e13}k_n - \frac{(a_n-1)a_{n+1}}{a_{n+1}-a_n+1}\ \leqslant\ 1.\end{equation}
We obtain a formula for $k_n$ depending on whether $a_n(a_{n+1}-1)/(a_{n+1}-a_n-1)$ is an integer or not. 

Case 1: if $a_n(a_{n+1}-1)/(a_{n+1}-a_n-1)\in \mathbb{N}$, then 
$$k_n\ =\  \frac{a_n(a_{n+1}-1)}{a_{n+1}-a_n-1}-1.$$
Hence, \eqref{e13} is equivalent to
$$\frac{a_n(a_{n+1}-1)}{a_{n+1}-a_n-1} - \frac{(a_n-1)a_{n+1}}{a_{n+1}-a_n+1}\ \leqslant\ 2.$$
Equivalently,
$$a_{n+1}^2-(4a_n-1)a_{n+1}+(a_n^2+a_n-2)\ \geqslant\ 0,$$
giving
$$a_{n+1}\ \geqslant\ 2a_n+\frac{\sqrt{3}}{2}\sqrt{4a_n^2-4a_n+3}-\frac{1}{2}.$$

Case 2: if $a_n(a_{n+1}-1)/(a_{n+1}-a_n-1)\notin \mathbb{N}$, then
$$k_n\ =\ \left\lfloor\frac{a_n(a_{n+1}-1)}{a_{n+1}-a_n-1}\right\rfloor\ =\ a_n + \left\lfloor\frac{a_n^2}{a_{n+1}-a_n-1}\right\rfloor.$$
Hence, \eqref{e13} is equivalent to
$$\left(a_n + \left\lfloor\frac{a_n^2}{a_{n+1}-a_n-1}\right\rfloor\right)-\left(\frac{(a_n-1)^2}{a_{n+1}-a_n+1}+(a_n-1)\right)\ \leqslant\ 1,$$
giving
$$\left\lfloor\frac{a_n^2}{a_{n+1}-a_n-1}\right\rfloor\ \leqslant\ \frac{(a_n-1)^2}{a_{n+1}-a_n+1}.$$
\end{proof}

\begin{cor}[Sufficient condition for uniqueness]\label{c3}
Let $(a_n)_{n=1}^\infty$ be increasing with $a_1\geqslant 2$ and $a_n\rightarrow\infty$. If
\begin{enumerate}
\item[i)] $a_{n+1}-2\geqslant a_n\geqslant 2$ for all $n$, and 
\item[ii)] for each $n\geqslant 1$, one of the following holds
\begin{enumerate}
\item[a)] $a_{n+1}-a_n-1$ divides $a_n^2$, and 
$$a_{n+1} \ \geqslant\ \frac{\sqrt{3}}{2}\sqrt{4a_n^2-4a_n+3}+2a_n-\frac{1}{2};$$ 
\item[b)] $a_{n+1}-a_n-1$ does not divide $a_n^2$, and 
$$
\left\lfloor\frac{a_n^2}{a_{n+1}-a_n-1}\right\rfloor\ \leqslant\ \frac{(a_n-1)^2}{a_{n+1}-a_n+1},
$$
\end{enumerate}
\end{enumerate}
then there exist unique $\theta\in (0,1]$ and $(b_n)_{n=1}^{\infty}$ such that 
\begin{equation}\label{e14}\sum_{n=1}^\infty \frac{1}{b_n} \ =\ \theta,\end{equation}
and for every $n\geqslant 1$, \begin{equation}\label{e15}a_n \ =\ G\left(\theta-\sum_{i=1}^{n-1}\frac{1}{b_i}\right).\end{equation}
\end{cor}

\begin{proof}
Theorem \ref{m1} guarantees the existence of $\theta$ and $(b_n)$. Suppose that there exists another pair $(\theta', (b_n'))$ different from $(\theta, (b_n))$. Then for some $N$, $b_N\neq b'_N$, both of which must verify \eqref{e8}. This contradicts Proposition \ref{p3}.
\end{proof}

Next, we establish a necessary condition for the uniqueness of $\theta$ and $(b_n)$ by requiring the inequalities
\begin{equation}\label{e18}\frac{a_{n}a_{n+1}}{a_{n+1}-a_{n}}-1-\frac{2a_{n}-1}{a_{n+1}-a_{n}}\ \leqslant\ b_{n}\ \leqslant\ \frac{a_{n}a_{n+1}}{a_{n+1}-a_{n}}\end{equation}
to determine exactly one solution $b_n$. Again, \eqref{e18} is slightly different from \eqref{e17} as we allow equalities, because the construction in the proof of Theorem \ref{m1} still works if equalities appear in finitely many \eqref{e11}.

\begin{prop}[Necessary condition for uniqueness]\label{p4}
Let $(a_n)_{n=1}^\infty$ be non-decreasing with $a_1\geqslant 2$ and $a_n\rightarrow\infty$. Suppose that there exist unique $\theta\in (0,1)$ and $(b_n)_{n=1}^\infty$ that satisfy \eqref{e14} and \eqref{e15}, then for all $n\geqslant 1$, we have
$$a_{n+1}\ \geqslant\ a_n+2,$$ 
$$(a_{n+1}-a_n) \mbox{ does not divide } a_na_{n+1},$$
and
\begin{equation}\label{e19}\left\lfloor \frac{a_n^2}{a_{n+1}-a_n}\right\rfloor\ <\ \frac{(a_n-1)^2}{a_{n+1}-a_n}.\end{equation}
\end{prop}

\begin{proof} That $a_{n+1}\geqslant a_n+2$ is due to the discussion at the beginning of this section. 
We find a sufficient and necessary condition for \eqref{e18} to have exactly one solution $b_n$. 
If $a_{n+1}-a_n$ divides $a_na_{n+1}$, then 
$$I_n\ :=\ \left[\frac{a_{n}a_{n+1}}{a_{n+1}-a_{n}}-1-\frac{2a_{n}-1}{a_{n+1}-a_{n}}, \frac{a_{n}a_{n+1}}{a_{n+1}-a_{n}}\right]$$ contains at least two integers because
$$\frac{a_{n}a_{n+1}}{a_{n+1}-a_{n}} - \left(\frac{a_{n}a_{n+1}}{a_{n+1}-a_{n}}-1-\frac{2a_{n}-1}{a_{n+1}-a_{n}}\right)\ > \ 1.$$
If $a_{n+1}-a_n$ does not divide $a_na_{n+1}$, then 
the largest integer in $I_n$ is 
$$\left\lfloor\frac{a_{n}a_{n+1}}{a_{n+1}-a_{n}}\right\rfloor,$$
and $I_n$ contains exactly one integer if and only if 
$$\left\lfloor\frac{a_{n}a_{n+1}}{a_{n+1}-a_{n}}\right\rfloor - \left(\frac{a_{n}a_{n+1}}{a_{n+1}-a_{n}}-1-\frac{2a_{n}-1}{a_{n+1}-a_{n}}\right)\ <\ 1.$$
Equivalently, 
$$\left\lfloor \frac{a_n^2}{a_{n+1}-a_n}\right\rfloor\ <\ \frac{(a_n-1)^2}{a_{n+1}-a_n}.$$
This completes our proof. 
\end{proof}

\begin{cor}
Let $(a_n)_{n=1}^\infty$ be non-decreasing with $a_1\geqslant 2$ and $a_n\rightarrow\infty$. Suppose that there exist unique $\theta\in (0,1)$ and $(b_n)_{n=1}^\infty$ that satisfy \eqref{e14} and \eqref{e15}, then for all $n\geqslant 1$,
\begin{enumerate}
\item[i)] $a_{n+1}-a_n$ divides none of $(a_n-1)^2, a_n^2, a_na_{n+1}$;
\item[ii)] $3a_n< a_{n+1}$.
\end{enumerate}
\end{cor}

\begin{proof}
i) By Proposition \ref{p4}, $a_{n+1}-a_n$ does not divide $a_na_{n+1}$. By \eqref{e19}, $a_{n+1}-a_n$ does not divide $a_n^2$. Also by \eqref{e19}, $a_{n+1}-a_n$ does not divide $(a_n-1)^2$. Indeed, supposing otherwise, we have
$$\left\lfloor \frac{a_n^2}{a_{n+1}-a_n}\right\rfloor\ =\ \left\lfloor\frac{(a_n-1)^2+2a_n-1}{a_{n+1}-a_n}\right\rfloor\ =\ \frac{(a_n-1)^2}{a_{n+1}-a_n} + \left\lfloor\frac{2a_n-1}{a_{n+1}-a_n}\right\rfloor,$$
contradicting \eqref{e19}.

ii) We write \eqref{e19} as
$$\left\lfloor \frac{a_n^2}{a_{n+1}-a_n}\right\rfloor\ <\ \frac{a_n^2}{a_{n+1}-a_n} - \frac{2a_n-1}{a_{n+1}-a_n}.$$
Hence, 
$$\frac{2a_n-1}{a_{n+1}-a_n}\ <\  1,$$
which gives $a_{n+1}\geqslant 3a_n$. However, $a_{n+1}$ cannot be $3a_n$. Otherwise, we obtain from \eqref{e19} that
$$\left\lfloor\frac{a_n^2}{2a_n}\right\rfloor\ <\ \frac{(a_n-1)^2}{2a_n}.$$
By i), $\lfloor a_n^2/(2a_n)\rfloor = (a_n-1)/2$. Hence,
$$\frac{a_n-1}{2}\ <\ \frac{(a_n-1)^2}{2a_n}\ \Longrightarrow\ a_n < 1,$$
a contradiction. 
\end{proof}

\section{Applications to particular sequences}\label{seq}

In this section, we look at sequences $(a_n)$ of special forms and find $(b_n)$ that satisfies $\eqref{e17}$. We use specific sequences in \cite{Sl23} as examples.

\subsection{Geometric progressions}
Let $a, r\in \mathbb{N}$ with $a\geqslant 2$ and $r\geqslant 2$. Let $(a_n)$ be the sequence
$$a, ar, ar^2, ar^3, \ldots.$$
By Corollary \ref{c2}, $(a_n)$ can be obtained from the WGAA applied to some $\theta$.

If $r-1$ divides $a$, we have the sequence $b_n =  ar^n/(r-1) - 1$ satisfy $\eqref{e17}$ and 
$$\theta \ =\ \sum_{n=1}^\infty \frac{1}{ar^n/(r-1)-1}.$$
For example, take $a = 2, r = 3$ to have 
$$\begin{cases}
a_n & = \quad 2\cdot 3^{n-1} \quad (\seqnum{A008776}),\\
b_n & = \quad 3^{n}-1 \quad (\seqnum{A024023}),\\
\theta& \approx \quad 0.68215 \text{ (irrational due to \cite{Er48})}.
\end{cases}$$

If $r-1$ does not divide $a$, we have the sequence $b_n = \lfloor ar^n/(r-1)\rfloor$ satisfy $\eqref{e17}$ and 
$$\theta \ =\ \sum_{n=1}^\infty \frac{1}{\lfloor ar^n/(r-1)\rfloor}.$$
For example, take $a = 2, r = 4$ to have
$$\begin{cases}
a_n & = \quad 2^{2n-1} \quad (\seqnum{A004171}),\\
b_n & = \quad \lfloor 2^{2n+1}/3\rfloor = 2(4^n-1)/3 \quad (\seqnum{A020988}),\\
\theta& \approx \quad 0.63165 \text{ (irrational due to \cite{Er48})}.
\end{cases}$$

\subsection{Arithmetic progressions}
Let $a, d\in \mathbb{N}$ with $a\geqslant 2$ and $d\geqslant 1$. Let $(a_n)$ be the sequence
$$a, a + d, a + 2d, a + 3d, \ldots.$$
By Corollary \ref{c2}, $(a_n)$ cannot be obtained from the WGAA applied to some $\theta$. 

If $d$ divides $a^2$, then 
$$b_n \ =\ \frac{(a+(n-1)d)(a+nd)}{d} - 1\ =\ \frac{a^2}{d}+(2n-1)a+n(n-1)d-1.$$
satisfies \eqref{e17} and 
$$\theta \ =\ \sum_{n=1}^\infty\left( \frac{a^2}{d}+(2n-1)a+n(n-1)d-1\right)^{-1}.$$
For example, take $a = 2$ and $d = 1$ to have 
$$\begin{cases}
a_n & = \quad n+1,\\
b_n & = \quad  n^2+3n+1\quad (\seqnum{A028387}),\\
\theta& = \quad \pi\tan\left(\frac{\sqrt{5}\pi}{2}\right)/\sqrt{5}\ \approx\ 0.54625.
\end{cases}$$

If $d$ does not divide $a^2$, then
$$b_n \ =\ \left\lfloor \frac{(a+(n-1)d)(a+nd)}{d}\right\rfloor$$
satisfies \eqref{e17} and
$$\theta \ =\ \sum_{n=1}^\infty \left\lfloor \frac{(a+(n-1)d)(a+nd)}{d}\right\rfloor^{-1}.$$
For example, take $a = 3$ and $d=2$ to have 
$$\begin{cases}
a_n & = \quad 2n+1,\\
b_n & = \quad  2n^2+4n+1\quad (\seqnum{A056220}),\\
\theta& = \quad \left(-2-\sqrt{2}\pi\cot\left(\frac{\pi}{\sqrt{2}}\right)\right)/4\ \approx\ 0.34551.
\end{cases}$$

\subsection{Fibonacci sequence}
The Fibonacci sequence is defined as $F_0 = 0$, $F_1 = 1$, and $F_{n} = F_{n-1} + F_{n-2}$ for $n\geqslant 2$. Define $a_n = F_{n+1}$ for $n\geqslant 1$. Then 
$$\frac{1}{a_n} - \frac{1}{a_{n+1}}\ =\ \frac{1}{F_{n+1}} - \frac{1}{F_{n+2}}\ =\ \frac{F_n}{F_{n+1}F_{n+2}}.$$
Using \eqref{e17}, we choose $b_1 = 3$ and 
for $n > 1$, choose 
\begin{align*}
    b_n &\ =\ \left\lfloor \frac{F_{n+1}F_{n+2}}{F_n}\right\rfloor\\
    &\ =\ \left\lfloor \frac{F_nF_{n+1}+F_{n-1}F_{n+1}+F_n^2+F_{n-1}F_n}{F_n}\right\rfloor\\
    &\ =\ \left\lfloor \frac{F_nF_{n+1}+F_n^2+(-1)^{n}+F_n^2+F_{n-1}F_n}{F_n}\right\rfloor \mbox{ by the Cassini's identity}\\
    &\ =\ \begin{cases}F_{n+3} - 1&\mbox{ if }n\mbox{ is odd},\\ F_{n+3}&\mbox{ if }n\mbox{ is even}.\end{cases}
\end{align*}

\ \\
\end{document}